\documentclass[12pt]{amsart}
\usepackage[margin=1in]{geometry}

\usepackage{amsthm,amsmath,amssymb}
\usepackage{enumerate}
\usepackage{multicol}
\usepackage{float}
\usepackage{mathtools}

\theoremstyle{plain}
\newtheorem{thm}{Theorem}[section]
\newtheorem{lem}[thm]{Lemma}
\newtheorem{cor}[thm]{Corollary}
\newtheorem{conj}[thm]{Conjecture}

\theoremstyle{definition}
\newtheorem{defi}[thm]{Definition}
\newtheorem{ex}[thm]{Example}
\newtheorem{rmk}[thm]{Remark}

\newcommand{\cL}{\mathcal{L}}

\newcommand{\cI}{\mathcal{I}}
\newcommand{\cT}{\mathcal{T}}
\newcommand{\cM}{\mathcal{M}}
\newcommand{\cC}{\mathcal{C}}

\usepackage{pgf,tikz}
\usetikzlibrary{arrows,shapes}
\usetikzlibrary{decorations.pathreplacing}
\usepackage{tkz-berge}
\usepackage{ytableau}
\usepackage{mathdots}

\usepackage{todonotes}

\DeclareMathOperator{\sm}{small}

\DeclarePairedDelimiter{\floor}{\lfloor}{\rfloor}

\title{Extremal Threshold Graphs for Matchings and Independent Sets}
\author{L. Keough}
\author{A.J. Radcliffe}

\begin{document}

\maketitle

\begin{abstract}
    Many extremal problems for graphs have threshold graphs as their extremal examples. For instance the current authors proved that for fixed $k\ge 1$, among all graphs on $n$ vertices with $m$ edges, some threshold graph has the fewest matchings of size $k$; indeed either the lex graph or the colex graph is such an extremal example. In this paper we consider the problem of maximizing the number of matchings in the class of threshold graphs. We prove that the minimizers are what we call \emph{almost alternating threshold graphs}.

We also discuss a problem with a similar flavor: which threshold graph has the fewest independent sets. Here we are inspired by the result that among all graphs on $n$ vertices and $m$ edges the lex graph has the most independent sets.
\end{abstract}

\section{Introduction} 
\label{sec:introduction}

In \cite{Matchings} the current authors proved the following theorem about the number of matchings in a graph with $n$ vertices and $e$ edges. 

\begin{thm}\label{thm:KR}
	If $G$ is a graph on $n$ vertices and $e$ edges then 
	\[
		m_k(G) \ge \min( m_k(\cL(n,e)), m_k(\cC(n,e))),
	\]
	where $\cL(n,e)$ and $\cC(n,e)$ are the lex graph and colex graph respectively and $m_k(G)$ is the number of matchings of size $k$ in $G$.
\end{thm}

The proof proceeds by showing that for any graph $G$ there is a threshold graph $T$ having the same number of vertices and edges as $G$ with $m_k(G)\ge m_k(T)$. Thus some $k$-matching minimizer among all graphs of given order and size is a threshold graph. 

Similarly, it is a  standard result, indeed a straightforward corollary of the Kruskal-Katona theorem (\cite{Kruskal}, \cite{Katona}), that the lex graph $\cL(n,e)$ has the largest number of independent sets for its order and size. Cutler and the second author \cite{CR} discuss this theorem and also prove that for any graph $G$ there is a threshold graph on the same number of vertices and edges having at least as many independent sets. This proof (and that of Theorem \ref{thm:KR}) proceed by ``compressing'' the graph so that it looks more and more like a threshold graph, at every stage not increasing the number of matchings and not decreasing the number of independent sets. 

These examples raise the interesting question of how many matchings we can preserve and still be threshold. Similarly, how many independent sets are forced to exist in a threshold graph? We solve these problems in this paper.

The answer to the second question is simple to state; in Section \ref{sec:ind} we prove that the colex graph has the fewest independent sets among threshold graphs on a given number of vertices and edges. The first question has a more subtle answer. In Section \ref{sec:results} we introduce the notion of an \emph{almost alternating} threshold graph, and in Section \ref{sec:maxmatch} we show that the threshold graphs on $n$ vertices with $e$ edges having the most matchings are exactly the almost alternating threshold graphs.


\section{Definitions and Results}\label{sec:results}

Threshold graphs have made many  appearances as the solutions to extremal problems for graphs of fixed order and size. (See for instance \cite{KatonaAhlswede} for the case $k=2$ of Theorem~\ref{thm:KR}; \cite{PeledEtAl} for  a discussion of the cases  of equality, and \cite{ChHam} for a very early application of threshold graphs to set packing problems.) 

There are many equivalent definitions for threshold graphs. (See \cite{OrangeBook} for an extensive discussion of many aspects of threshold graphs.)  One that will be particularly useful to us is as follows.  
\begin{defi}
A \emph{threshold graph} is a graph that can be constructed from a single vertex by adding vertices one at a time that are either isolated or dominating.  Let $\cT_{n,e}$ be the set of all threshold graphs having $n$ vertices and $e$ edges.
\end{defi}

Using this definition any binary string of finite length can be used as instructions to construct a threshold graph.  Letting 1 be the code for a dominating vertex and  0 be the code for an isolated vertex, construct the threshold graph by reading the binary string from right to left.  (With this convention the first digit of the code corresponds to the last vertex added and is the ``most significant".  An initial $1$ is a dominating vertex in the graph, for instance.) The threshold graph with binary string $\sigma$ will be denoted $T(\sigma)$.  We will refer to $\sigma$ as the code of the graph.

\begin{ex}  The graph in Figure \ref{fig:threshold} is threshold with code $0010010$.
\label{thresholdex}
\begin{figure}[!ht]
\begin{center}
\begin{tikzpicture}
    \coordinate (0) at (0,0) ;
    \coordinate (1) at (1,0) ;
    \coordinate (2) at (2,0) ;
    \coordinate (3) at (3,0) ;
    \coordinate (4) at (4,0) ;
    \coordinate (5) at (5,0) ;
    \coordinate (6) at (6,0) ;

	\draw (0,-.5) node {$0$}; 
	\draw (1,-.5) node {$0$}; 
	\draw (2,-.5) node {$1$}; 
	\draw (3,-.5) node {$0$}; 
	\draw (4,-.5) node {$0$}; 
	\draw (5,-.5) node {$1$}; 
	\draw (6,-.5) node {$0$}; 
   
    \fill (0) circle (3pt) ;
    \fill (1) circle (3pt) ;
    \fill (2) circle (3pt) ;
    \fill (3) circle (3pt) ;
    \fill (4) circle (3pt) ;
    \fill (5) circle (3pt) ;
    \fill (6) circle (3pt) ;

   \tikzstyle{EdgeStyle}=[bend left]

   \Edge(2)(6)
   \Edge(2)(5)
   \Edge(2)(4)
   \Edge(2)(3)
		
   \Edge(5)(6)
\end{tikzpicture}

\caption{$T(0010010)$}
\label{fig:threshold}
\end{center}
\end{figure}
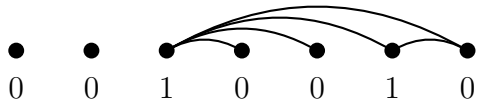

\end{ex}

Note that $T(0010010)$ and $T(0010011)$ are the same graph because whether the first vertex added to the graph is isolated or dominating does not affect the outcome.  To deal with this lack of uniqueness, a $*$ will be used to denote the first (rightmost) vertex as it can be read as either a 0 or a 1. With this convention the graph in Figure \ref{fig:threshold} will be denoted  $T(001001*)$.

One of the substructures we are concerned with is matchings.  A \emph{matching} is a set of independent edges.  Let $\cM(G)$ denote the set of matchings in $G$ and $m(G) = |\cM(G)|$.  

\begin{ex}
Let $G=T(001001*)$, shown in Figure \ref{fig:ThresholdMatching}.  The bold edges $v_5v_3$ and $v_2v_1$ form a matching of size 2, and there is one other $2$-matching. 
The empty set and single edges are also matchings.  Thus $m(G)=8$. 
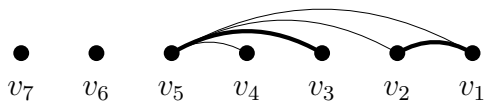
\begin{figure}[!ht]
\begin{center}
\begin{tikzpicture}
    \coordinate (0) at (0,0) ;
    \coordinate (1) at (1,0) ;
    \coordinate (2) at (2,0) ;
    \coordinate (3) at (3,0) ;
    \coordinate (4) at (4,0) ;
    \coordinate (5) at (5,0) ;
    \coordinate (6) at (6,0) ;

	\draw (0,-.5) node {$v_7$}; 
	\draw (1,-.5) node {$v_6$}; 
	\draw (2,-.5) node {$v_5$}; 
	\draw (3,-.5) node {$v_4$}; 
	\draw (4,-.5) node {$v_3$}; 
	\draw (5,-.5) node {$v_2$}; 
	\draw (6,-.5) node {$v_1$};  
   
    \fill (0) circle (3pt) ;
    \fill (1) circle (3pt) ;
    \fill (2) circle (3pt) ;
    \fill (3) circle (3pt) ;
    \fill (4) circle (3pt) ;
    \fill (5) circle (3pt) ;
    \fill (6) circle (3pt) ;

   \tikzstyle{EdgeStyle}=[bend left,ultra thin]

   \Edge(2)(6)
   \Edge(2)(5)
   
   \Edge(2)(3)

   \tikzstyle{EdgeStyle}=[bend left, ultra thick]
   \Edge(5)(6)
   \Edge(2)(4)
   
\end{tikzpicture}
\caption{A matching in $G = T(001001*)$.}
\label{fig:ThresholdMatching}
\end{center}
\end{figure}

\end{ex}

Our analysis of the maximum matchings problem for threshold graphs leads us to consider graphs that satisfy a certain condition on their codes. 
\begin{defi}
A threshold graph is \emph{alternating} if its code is a block of $1$'s or a block of $0$'s, followed by an alternating string of $1$'s and $0$'s.
\end{defi}

There are many values of $n$ and $e$ for which there is no alternating threshold graph on $n$ vertices with $e$ edges.  We introduce an extended class that covers the remaining possibilities.
 
\begin{defi}
Let $a$  denote $01$ and $b$ denote $10$.  A threshold graph is \emph{almost alternating} if its code can be written as a block of 1's or 0's followed by a string of $a$'s and $b$'s.  We may consider the $*$ to be a $0$ or a $1$ and use it as part of an $a$ or a $b$.  We are also permitted to keep the star separate.  In the event that the $*$ appears in the $a,b$ representation of the graph, then the almost alternating graph will be called \emph{starred}.  Otherwise we call the graph \emph{unstarred}.  Also, we distinguish between \emph{small} almost alternating threshold graphs whose starting block consists only of $0$'s (or is empty) and  \emph{large} ones, whose starting block consists of $1$'s. Clearly alternating graphs are almost alternating.
\end{defi}

\begin{ex}  The following are examples of codes of almost alternating threshold graphs.
\begin{enumerate}
\item[(i)]  $00001011001*=000aaba*$ is small and starred 
\item[(ii)]  $11101100*=111aba$ is  large and unstarred
\item[(iii)]  $0010101*=0aaa*=00bbb$ is small and alternating
\end{enumerate}

The last of these three examples demonstrates that representations of almost alternating threshold graphs using $a$'s and $b$'s are not unique when the graph is alternating.  In fact, this is the only case in which there are two representations using $a$'s and $b$'s.

\end{ex}

Our first main theorem, Theorem \ref{thm:MaxMatchings}, gives results about maximizing $m(G)$.

\begin{thm}\label{thm:MaxMatchings}
If $A$ is an almost alternating threshold graph, and $G$ is a threshold graph with the same number of vertices and edges then $m(G)\leq m(A)$ and if $G$ is not almost alternating then $m(G) < m(A)$.
\end{thm}

The proof of Theorem \ref{thm:MaxMatchings} appears in Section \ref{proof:maxmatchings}.  

\begin{rmk}
Typically, there are many almost alternating graphs having $n$ vertices and $e$ edges all having the same number of matchings, since neither the number of matchings nor the number of edges depends on the order of $a$'s and $b$'s (see Corollary \ref{cor:SameNumber}).   Also, it is important to note that it is not the case that $G\in\cT_{n,e}$ attains the maximum number of matchings of size $k$ in $\cT_{n,e}$ \emph{if and only if} $G$ is almost alternating.  For instance, all graphs in $\cT_{n,e}$ have $e$ $1$-matchings.  Also, if the relevant almost alternating graph has no $k$-matchings then we can't have strict inequality.  For example consider $G = T(1000111*)$ and $G' =T(1010100*)$. While $G'$ is almost alternating, $G$ is not, and they each have $8$ vertices, $13$ edges, and no matchings of size $4$.
\end{rmk}

In another direction we consider the problem of minimizing the number of independent sets in threshold graphs.  An \emph{independent set} is a subset of the vertices with no edge between any pair.  Let $\cI(G)$ denote the set of independent sets in $G$ and $i(G) = |\cI(G)|$.  Similarly, let $\cI_k(G)$ denote the set of independent sets in $G$ of size $k$ (those having $k$ vertices) and let $i_k(G) = |\cI_k(G)|$.

The \emph{colex graph} can be defined using the colex ordering.  Let $[n] = \{1,2,\dots, n\}$.  We define the \emph{colex ordering} by $A<_C B$ if and only if $\max(A\Delta B) \in B$ where $A\Delta B$ denotes the symmetric difference between $A$ and $B$.  Restricting this ordering to $\binom{[n]}{2}$, the subsets of $[n]$ of size $2$, we get an ordering on $E(K_n)$, the edge set of the complete graph on $n$ vertices.  The first few edges in this ordering are 
\[ \{1,2\}, \{1,3\}, \{2,3\}, \{1,4\}, \{2,4\}, \{3,4\}, \{1,5\}, \dots\]
The colex graph having $n$ vertices and $e$ edges has vertex set $[n]$ and edge set the first $e$ edges in the colex ordering.  Throughout we write $\cC(n,e)$ to mean the colex graph having $n$ vertices and $e$ edges. The colex graph is a threshold graph. In Section \ref{sec:colex}  we provide an alternative characterization of colex graphs using the binary code representation.  The following theorem establishes that the colex graph minimizes independent sets among all threshold graphs.

\begin{thm}\label{thm:MinIndSets}
Let $G$ be a threshold graph with $n$ vertices and $e$ edges  and let $\cC(n,e)$ be the colex graph having $n$ vertices and $e$ edges.  If $G$ is not $\mathcal{C}(n,e)$ then for all $k$,
	\[i_k(\mathcal{C}(n,e)) \leq i_k(G) \quad\text{and}\quad i(\mathcal{C}(n,e)) < i(G).\]
\end{thm}

We prove Theorem \ref{thm:MinIndSets} in Section \ref{sec:localmovesindsets}.

\section{Maximizing Matchings}\label{sec:maxmatch}

\subsection{Existence of Almost Alternating Threshold Graphs}  Given integers $n\geq 0$ and $e\leq \binom{n}{2}$, it is not immediately obvious that there exists an almost alternating graph on $n$ vertices with $e$ edges.  We prove this in Corollary \ref{cor:findingcode}.  We first prove some preliminary results.

Given $n\ge 0$ we write $\sm(n)$ for the maximum number of edges in a small almost alternating threshold graph.

\begin{lem}
    For all $n\ge 0$ there is a unique almost alternating graph that is simultaneously the small graph with the most edges and the large graph with the fewest edges. This graph is in fact alternating and the number of edges is
    \[
        \sm(n) = \floor[\Big]{\frac{n^2}4} .
    \]
\end{lem}
\begin{proof}
    It is clear from the definition of the letters $a$ and $b$ that the code for the small graph with the most edges has no initial string of $0$'s, has a string of $b$'s of length $\lfloor n/2\rfloor$, and then terminates with $*$ if $n$ is odd. Similarly the code for the smallest large graph starts with a $1$ and then has a sequence of $a$'s of length $\lfloor (n-1)/2 \rfloor$, terminating with a $*$ if $n$ is even. It is easy to check that 
    \[
        bbb\dots bbb = 1aaa\dots aaa{*} \qquad \text{and}\qquad bbb\dots bbb* = 1aaa\dots aaa
    \]
    when the sequences have the same lengths. Counting edges according their left hand end in the code we get for $n$ even,
    \[
        \sm(n) = e(T(bbb\dots bbb)) = (n-1)+(n-3)+\cdots+1 = \frac{n^2}{4},
    \]
    and for $n$ odd,
    \[
        \sm(n) = e(T(bbb\dots bbb{*})) = (n-1)+(n-3)+\cdots+2 = \frac{n^2-1}4.
    \]
\end{proof}

\begin{lem}\label{lem:edgecount}
Defining $s(\alpha,\beta)$ (respectively $s^*(\alpha,\beta)$) to be the number of edges in an  unstarred (respectively starred) small graph with $\alpha$ $a$'s and $\beta$ $b$'s,  then
\[ s(\alpha,\beta) = (\alpha+\beta)^2-\alpha \qquad \text{and} \qquad s^*(\alpha,\beta) = (\alpha+\beta)^2 + \beta.\]
In particular, the number of edges is independent of the positions of the $a$'s and $b$'s.
\end{lem}
\begin{proof}
Consider a  small unstarred graph with $\alpha$ $a$'s and $\beta$ $b$'s.   Let  $A$ be the set of positions of $a$'s and $B$ be the set of positions of $b$'s where we index starting with the rightmost letter as $0$. We count the number of edges in this graph by counting the number of edges each vertex labeled $1$ has to its right.  An $a$ in the $i^{th}$ position is incident to $2i$ edges going to the right.  A $b$ in the $i^{th}$ position has $2i+1$ edges going to the right.  Thus, 
\[
s(\alpha,\beta) = \sum_{i\in A} 2i + \sum_{i\in B} (2i+1) = \beta + 2 \sum_{i=0}^{\alpha+\beta-1} i
= (\alpha+\beta)^2 -(\alpha+\beta) + \beta 
=(\alpha+\beta)^2 -\alpha.
\]

For a  small starred graph the number of edges is the number of edges in the corresponding  small unstarred plus the one edge to the star for each letter.  So
\[ s^*(\alpha,\beta) = (\alpha+\beta)^2-\alpha + (\alpha+\beta)= (\alpha+\beta)^2 + \beta.\]
\end{proof}

\begin{cor}\label{cor:codesmall}
Suppose integers $n$ and $e$ are such that $n\geq 0$ and $0\leq e \leq \floor[\big]{\frac{n^2}4}$. Consider where $e$ fits in the sequence 
\[0\le 1\le 2 \le 4 \le 6 \le 9 \le 12 \le \cdots \le k^2 \le k(k+1) \le (k+1)^2 \le \cdots .\]
If $k^2\leq e \leq k(k+1)$ then there is a small starred graph with $n$ vertices and $e$ edges having $\beta = e-k^2$ copies of $b$ and $\alpha =k-\beta$ copies of $a$. Similarly if $k(k+1)\leq e \leq (k+1)^2$ then there is a small unstarred graph with $n$ vertices and $e$ edges having $\alpha = (k+1)^2 - e$ copies of $a$ and $\beta = (k+1)-\alpha$ copies of $b$.  These values of $\alpha$ and $\beta$ are unique except where $e =k^2$ or $e=k(k+1)$ in which case there are exactly two possibilities, $00\cdots 0aa\cdots aaa$ or $00\cdots 00bb\cdots bbb*$, which in fact encode the same threshold graph.
\end{cor}
\begin{proof}
Immediate from Lemma \ref{lem:edgecount}, noting that we add a block of $0$'s at the beginning of the code of size $n-2\alpha-2\beta$ in the unstarred case and $n-2\alpha-2\beta-1$ in starred case.
\end{proof}

\begin{cor}\label{cor:findingcode}
Given integers $n\ge 0$ and $0\le e \le \binom{n}{2}$ there is an almost alternating threshold graph on $n$ vertices with $e$ edges.
 Moreover the number of $a$'s and the number of $b$'s appearing in any binary representation of its code is determined by $e$, unless $e$ is of the form $k^2$, $k(k+1)$, $\binom{n}{2} - k^2$, or $\binom{n}{2} - k(k+1)$ for some $k<n/2$.  In these cases there are a different number of $a$'s and $b$'s, but both represent the same underlying code.
\end{cor}
\begin{proof}
By Corollary \ref{cor:codesmall} there exists an almost alternating graph on $n$ vertices with $e$ edges when $0\le e \le \floor[big]{\frac{n^2}{4}}$ where the number of $a$'s and $b$'s is determined by $e$.  Since the complement of an almost alternating graph is almost alternating we can use these codes to find codes for all $e\ge \binom{n}{2} - \frac{n^2}{4}$.  

When $e$ is between $\binom{n}{2} - \left\lfloor\frac{n^2}{4}\right\rfloor$ and  $\left\lfloor\frac{n^2}{4}\right\rfloor$ we can find an almost alternating threshold graph in two seemingly different ways: by using Corollary \ref{cor:codesmall} or by complementing the graph on $\binom{n}{2}-e$ given by Corollary \ref{cor:codesmall}.  However, these result in the same code.  In fact, if $n=2s$ then we consider graphs that have between $s^2-s$ and $s^2$ edges. Using Corollary \ref{cor:codesmall} yields an unstarred code having $s = n/2 = \alpha + \beta$ and thus these graphs have no initial block of $0$'s or $1$'s.  Moreover, if $e=s(s-1)+k$ then the complement graph has $e'=s^2-k$ edges.  By Corollary \ref{cor:codesmall} this switches $\alpha$ with $\beta$ and complementing switches  $\alpha$ and $\beta$ back resulting in the other code.  This works similarly when $n$ is odd.  In that case we consider starred graphs and the star is unaffected by complementing.  Therefore, the values of $\alpha$ and $\beta$ are unique except in the alternating case.
\end{proof}

\subsection{Characterizing Almost Alternating Graphs}
The goal for this section is to establish an alternative characterization for almost alternating graphs that will help us prove Theorem \ref{thm:MaxMatchings}.

\begin{defi}
We will say a threshold graph code $\sigma$ has a \emph{bracketed 0-string} if it contains a substring with at least three 0's with 1's on both ends (where the 1 on the right end may be the $*$).  Similarly, $\sigma$ has a \emph{bracketed 1-string} if it contains a string of at least three 1's with 0's on both ends (where the $0$ on the right end may be the $*$).  We say $\sigma$ contains a \emph{bracketed string} in either case.
\end{defi}

\begin{ex}
The following codes have bracketed strings, which are highlighted.  The second is as short as it could be.
\begin{enumerate}
\item $100\mathbf{1000001}0101*$
\item $\mathbf{0111*}$
\end{enumerate}
\end{ex}

\begin{defi} We will say that a threshold graph code $\sigma$ has a \emph{separation issue} if it has two pairs of repeated digits separated by a  substring of odd length, with the first pair preceded by the opposite digit and and the last pair not ending the code.
\end{defi}

\begin{ex}  The following codes have separation issues. The two pairs of repeated digits are highlighted. The second example is as short as it could be.
\begin{enumerate}
\item $ 11110101\mathbf{00}0010101\mathbf{00}011*$
\item  $0\mathbf{11}0\mathbf{11}*$
\end{enumerate}

\end{ex}

In the next lemma we prove that these are the only obstacles to a graph being almost alternating.

\begin{lem}  
\label{almostalternating} A code that has neither a separation issue nor a bracketed string is the code of an almost alternating graph.
\end{lem}
\begin{proof}
Let $G = T(\sigma)$ be a threshold graph and suppose that $\sigma$ does not have a bracketed string or a separation issue.  If $\sigma$ is alternating then clearly $\sigma$ is almost alternating.  
Otherwise $\sigma_0=0^k$ or $\sigma_1=1^k$ appears somewhere in $\sigma$ for some $k\geq 2$.  
If there is no opposite digit to the left of $\sigma_i$ for $i=0$ or $1$ then $\sigma_i$ can be considered part of the beginning block.  
So suppose there is an opposite digit to the left of $\sigma_i$.  
If $k>2$ then we have the string $01^k0$ or $10^k1$ possibly using the $*$.  
This is a contradiction as we assumed there is no bracketed string.
Assuming $k=2$, if it is not possible to write the code as a beginning block of 0's or 1's followed by $a$'s and $b$'s then it must be the case that there are two pairs of repeated digits separated by a string of odd length.  Additionally, the first pair of repeated digits must be preceded by the opposite digit, else we could consider it part of the beginning block.  
The last pair of the repeated digits can not end the code since we can use the $*$ or not when finding the representation in $a$'s and $b$'s.
So $G$ has a separation issue, which is a contradiction.
 Thus, the code of $G$ has been written as a starting block followed by a string of $a$'s and $b$'s.\end{proof}

\subsection{Local Moves for Matchings}\label{sec:localmovesmatchings}
To prove Theorem \ref{thm:MaxMatchings} we will make local switches in the code of a threshold graph that result in a graph having at least as many matchings, without changing the number of vertices or edges.  The first move, the $ab$-switch, will preserve the total number of matchings.
The other two local moves in this section will together show that threshold graphs which are not almost alternating do not attain the maximum number of matchings.

First, we present a simple lemma that will be used repeatedly.
\begin{lem}
\label{abvertedge} Let $\sigma,\tau,$ and $\rho$ be (possibly empty) binary strings.  Suppose that $G=T(\sigma01\tau10\rho)$ has $n$ vertices and $e$ edges. The graph $G'=T(\sigma 10 \tau 01 \rho)$ is also a threshold graph on $n$ vertices having $e$ edges.
\end{lem}
\begin{proof}
	Straightforward.
\end{proof}

\subsubsection{The $ab$-switch}

The next lemma will show that if we replace $ab = 0110$ with $ba = 1001$ in the code of a threshold graph the number of matchings of each size are preserved.  If $G'$ is obtained from $G$ by replacing  $ab$ with  $ba$ or vice versa, we will say that we have performed an \emph{$ab$-switch}.

\begin{lem}  
\label{abmatchings} Let $\sigma$ and $\rho$ be (possibly empty) binary strings. Consider $G=T(\sigma a b\rho)$ and $G'=T(\sigma b a\rho)$.  Then $G$ and $G'$ have the same number of vertices and edges. Moreover, $m(G) = m(G')$ and $m_k(G) = m_k(G')$ for all $k$.
\end{lem}
\begin{proof}
By Lemma~\ref{abvertedge} we know that $G$ and $G'$ have the same number of vertices and edges.  In fact, up to isomorphism, $G$ and $G'$ differ only by one edge. Figure  \ref{fig:0110} demonstrates the difference between $G$ and $G'$.  On the left the subgraph of $G$ induced by the vertices associated to the $ab =0110$ is shown and on the right is the subgraph of $G'$ induced by the same subset of the vertices.  Note $G'$ can be obtained from $G$ by removing edge $yw$ and adding edge $xv$.

\begin{figure}[!ht]
\begin{center}
\begin{tikzpicture}
\coordinate (y) at (0,0);
\coordinate (x) at (0,2);
\coordinate (w) at (2,0);
\coordinate (v) at (2,2);
\coordinate (y2) at (4,0);
\coordinate (x2) at (4,2);
\coordinate (w2) at (6,0);
\coordinate (v2) at (6,2);

\draw (-.5,0) node {$y$};
\draw (-.5,2) node {$x$};
\draw (2.5,0) node {$w$};
\draw (2.5,2) node {$v$};

\draw (3.5,0) node {$y$};
\draw (3.5,2) node {$x$};
\draw (6.5,0) node {$w$};
\draw (6.5,2) node {$v$};

    \fill (y) circle (3pt) ;
    \fill (x) circle (3pt) ;
    \fill (w) circle (3pt) ;
    \fill (v) circle (3pt) ;
    \fill (y2) circle (3pt) ;
    \fill (x2) circle (3pt) ;
    \fill (w2) circle (3pt) ;
    \fill (v2) circle (3pt) ;

   \Edge(y)(w)
   \Edge(v)(w)
   \Edge(y)(v)
   
   \Edge(x2)(v2)
   \Edge(y2)(v2)
   \Edge(w2)(v2)

\end{tikzpicture}
\caption{The subgraph induced by $0110$ and the subgraph induced by $1001$.}
\label{fig:0110}
\end{center}
\end{figure}
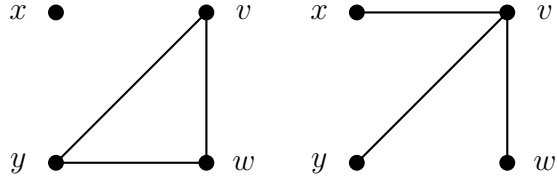

 In $G$, the code $0110$ corresponds to the vertices $xvwy$, i.e., $x$ is the left $0$, $y$ is the right $0$ and $v$ and $w$ are the left and right $1$'s, respectively.  
 In $G'$, on the other hand, $1001$ corresponds to $vxyw$.  That is, the vertex labels in the diagram are associated to a digit in the code and move with the digit when we make switches in the code.  Because of this, edges to the rest of the graph are the same in $G$ and $G'$. 
  
We will construct two injections: first $\phi: \cM(G)\to \cM(G')$  then $\phi':\cM(G') \to \cM(G)$. We first define a replacement function $r:E(G) \to E(G')$.  Let
 $$r(e)=
 \begin{cases}
 xv  &\text{ if } e=yw\\
 e\Delta\{x,y\} &\text{ if } x\in e\\
 e\Delta\{v,w\} &\text{ if } v\in e, w\notin e, y\notin e\\
 e  &\text{ otherwise}
 \end{cases}.$$  
 
We claim that any for edge $e$ in $E(G)$ we have $r(e)$ in $E(G')$.  Clearly $xv\in G'$. If $xc$ is some edge in $E(G)$ then $c$ must correspond to some $1$ to the left of $x$ in the code and so $c$ is also adjacent to $y$ in $G$.  Since $c\neq w$ we know $yc \in E(G')$.  
 Similarly, if $e=vc\in E(G)$ then $c$ corresponds to a $1$ to the left of both $v$ and $w$ or $c$ is to the right of $v$.   
 If $c$ is a $1$ to the left then $cw\in E(G)$ and so $cw\in E(G')$ as $c\neq y$.  
 If $c$ is to the right of $v$ then $c$ is also adjacent to $w$ in $G$ as the code of $w$ is a $1$ and $c\neq w$ by assumption.  
 So $cw\in E(G')$ since $cw\in E(G)$ and $c\neq y$ by assumption. Therefore $r(e)\in E(G')$ for all $e\in E(G)$. 
 
 Now define $\phi: \cM(G) \to \cM(G')$ by
 $$\phi(M)=
 \begin{cases}
 M  &\text{ if } yw\notin M\\
 \{r(e): e\in M\}  &\text{ if } yw\in M
 \end{cases}.$$
 
We need to show two things: that $\phi(M)$ is a matching in $G'$ and that $\phi$ is an injection.  
Note that $\phi(M)\subseteq E(G')$ by the argument above. 
To prove $\phi(M)$ is a matching we are only concerned about the case where we switch edges, i.e., when $yw\in M$. 
 If $yw\in M$ then we replace $yw$ with $xv$ and resolve conflicts at $x$ and $v$ resulting in a matching.

Suppose $M_1$ and $M_2$ are matchings such that $\phi(M_1)=\phi(M_2)$.  For any matching $M$ the edge $xv$ in $\phi(M)$ if and only if $yw$ is in $M$.  Thus, if $xv\notin \phi(M_1) = \phi(M_2)$ then $yw\notin M_1$ and $yw\notin M_2$ putting us in the first case so that $M_1=\phi(M_1)=\phi(M_2)=M_2$.  Suppose $xv\in \phi( M_1) = \phi(M_2)$.  Then $yw\in M_1 \cap M_2$ putting us in the second case.  If $yc\in \phi(M_1) = \phi(M_2)$ for some $c$ then $xc\in M_1\cap M_2$ and if $wd\in \phi(M_1)=\phi(M_2)$ for some $d$ then $vd\in M_1\cap M_2$.  Additionally, all other edges were not moved.  Thus, $M_1=M_2$.

Therefore, $\phi: \cM(G) \to \cM(G')$ is an injection and $m(G)\leq m(G')$.  Moreover, the injection preserves the size of the matching and so $m_k(G) \leq m_k(G')$ for all $k$.

Similarly define $r': E(G')\to E(G)$ by
$$r'(e) = 
\begin{cases}
yw &\text{if $e=xv$}\\
e\Delta\{x,y\} &\text{if $y\in e$}\\
e\Delta\{v,w\} &\text{if $w\in e$ and $v\notin e$}\\
e &\text{otherwise}
\end{cases}$$
and $\phi': \cM(G')\to \cM(G)$ by
$$\phi'(M)=
\begin{cases}
M &\text{if $xv\notin M$}\\
\{r(e): e\in M\} &\text{if $xv\in M$}
\end{cases}.$$
By a similar argument $\phi'$ is an injection that preserves size and so $m(G')\leq m(G)$ and $m_k(G')\leq m_k(G)$.  Therefore, $m(G) = m(G')$ and $m_k(G) = m_k(G')$.
\end{proof}

\begin{cor}
\label{cor:SameNumber}
Every almost alternating threshold graph on $n$ vertices and $e$ edges has the same number of matchings in total and the same number of matchings of each size. 
\end{cor}
\begin{proof}
The proof of Corollary \ref{cor:findingcode} shows that given $n$ and $e$, the number of $a$'s and $b$'s and the length of the starting block is determined (modulo some values of $e$, but this does not affect this proof as the different encodings in $a$'s and $b$'s represent the same underlying code). Lemma \ref{abmatchings}   shows that $a$'s and $b$'s in the code of the threshold graph commute.  Thus, all almost alternating graphs that have the same beginning block (both in length and in digit) and have the same number of $a$'s and $b$'s have the same number of matchings.  So, all almost alternating graphs with $n$ vertices and $e$ edges have the same number of matchings and the same number of matchings of each size. 
\end{proof}

\subsubsection{Bracketed Strings}
The proof of the next lemma gives us our first local move that will  increase the number of matchings. 

\begin{lem} 
\label{bracketedstring}  Any graph on $n$ vertices with $e$ edges that has a bracketed string appearing in its code does not have the maximum number of matchings in $\cT_{n,e}$.  
\end{lem}

\begin{proof}
Let $G$ be a graph on $n$ vertices with $e$ edges that has a bracketed 1-string, say $G = T(\sigma 01^j01 \tau)$ for $\sigma$ and $\tau$ (possibly empty) binary strings and $j\geq 3$.  Let $G' = T(\sigma 101^{j-2} 01\tau)$.  
We claim $G'$ has $n$ vertices, $e$ edges and strictly more matchings than $G$.  By Lemma \ref{abvertedge}, $G'$ has the same number of vertices and edges as $G$.

Our attention will be focused on the subgraph associated to the bracketed 1-string.  Figure \ref{fig:Bracketed1StringA} shows the subgraph of $G$ induced by the vertices in the bracketed $1$-string on the left.  On the right is the subgraph of $G'$ induced by the same subset of the vertices.

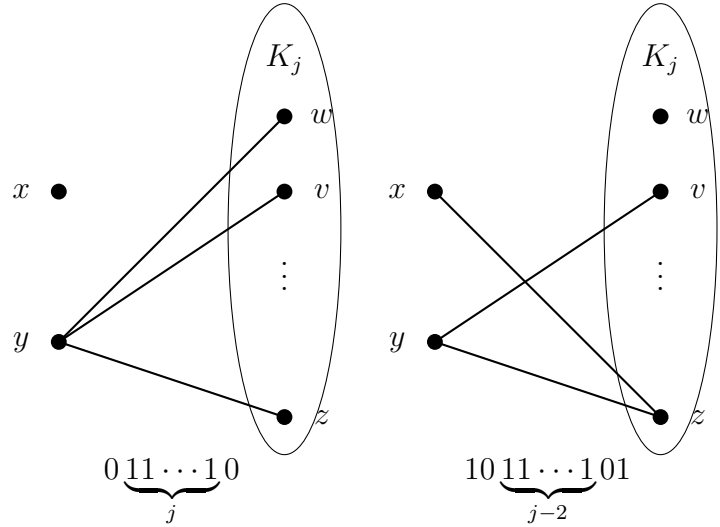
\begin{figure}[!ht]
\begin{center}
\begin{tikzpicture}
\coordinate (x) at (0,2);
\coordinate (y) at (0,0);
\coordinate (w) at (3,3);
\coordinate (v) at (3,2);
\coordinate (z) at (3,-1);

\draw (3,1.5) ellipse (.75cm and 3cm);
\draw (8,1.5) ellipse (.75cm and 3cm);

\draw (3,3.75) node {$K_j$};
\draw (8,3.75) node {$K_{j}$};

\draw (3,1) node {$\vdots$};
\draw (8,1) node {$\vdots$};

\draw (-.5,2) node {$x$};
\draw (-.5,0) node {$y$};
\draw (3.5,3) node {$w$};
\draw (3.5, 2) node {$v$};
\draw (3.5, -1) node {$z$};

\fill (x) circle (3pt);
\fill (y) circle (3pt);
\fill (w) circle (3pt);
\fill (z) circle (3pt);
\fill (v) circle (3pt);

\Edge (y)(z);
\Edge(y)(v);
\Edge(y)(w);

\draw (1.5,-2) node {$0\underbrace{11\cdots 1}_j0$};

\coordinate (x2) at (5,2);
\coordinate (y2) at (5,0);
\coordinate (w2) at (8,3);
\coordinate (v2) at (8,2);
\coordinate (z2) at (8,-1);

\draw (4.5,2) node {$x$};
\draw (4.5, 0) node {$y$};
\draw (8.5, 3) node {$w$};
\draw (8.5, 2) node {$v$};
\draw (8.5, -1) node {$z$};

\fill (x2) circle (3pt);
\fill (y2) circle (3pt);
\fill (w2) circle (3pt);
\fill (z2) circle (3pt);
\fill (v2) circle (3pt); 

\Edge (y2)(z2);
\Edge(y2)(v2);
\Edge(x2)(z2);

\draw (6.5, -2) node {$10\underbrace{11\cdots 1}_{j-2}01$};

\end{tikzpicture}
\caption{Removing a bracketed 1-string}
\label{fig:Bracketed1StringA}
\end{center}
\end{figure}

Let the vertex associated to the first 0 in the bracketed 1-string in $G$ be called $x$ and the vertex associated to the second 0 in the bracketed $1$-string be called $y$.  The vertices associated to the $1$'s in the bracketed 1-string in $G$ form a clique.  In Figure \ref{fig:Bracketed1StringA} the vertices inside the oval form this clique.  Let $w$ be the vertex associated to the rightmost $1$ in the bracketed 1-string and $z$ be the vertex associated to the leftmost $1$.    
Note that $V(G) = V(G')$ and $E(G')= E(G)-yw+xz$. To prove that $G'$ has strictly more matchings we will construct an injection from the matchings in $G$ to the matchings in $G'$ that is not a surjection.

First define $r(e): E(G)\to E(G')$ by
$$r(e)=
\begin{cases}
xz  &\text{if $e=yw$}\\
e\Delta \{z,w\} &\text{if $z\in e$ and $y\notin e$}\\
e\Delta \{x,y\} &\text{if $x\in e$}\\
e  &\text{otherwise}
\end{cases}.$$

One can show that $r(e)\in E(G')$ for all $e\in E(G)$.  Using $r(e)$ we define 
$$\phi(M)=
\begin{cases}
M  &\text{if $yw\notin M$}\\
\{r(e): e\in M\}  &\text{if $yw\in M$}
\end{cases}.$$

Suppose that $M_1$ and $M_2$ are two matchings in $G$ such that $\phi(M_1)=\phi(M_2)$.  Note that $xz\in \phi(M)$ if and only if $yw\in M$.  If $xz\notin \phi(M_1)=\phi(M_2)$ then $yw\notin M_1$ and $yw\notin M_2$ which means $M_1=\phi(M_1) = \phi(M_2) = M_2$.  Suppose that $xz\in \phi(M_1)=\phi(M_2)$.   Then $yw\in M_1\cap M_2$ and we use the second case.  We can ``undo" $r(e)$ to determine $M_1=M_2$.  In more detail, if $wc\in \phi(M_1)=\phi(M_2)$ for some $c$, then $zc\in M_1\cap M_2$ and if $yd\in \phi(M_1)=\phi(M_2)$ for some $d$ then $xd\in M_1\cap M_2$.  All other edges stay the same.  Thus $M_1=M_2$.  Consider the matching $M=\{xz, yv\}$.  Note that $M$ is not in the image of $\phi$ and so $m(G')>m(G)$.   Moreover, since $\phi$ preserves size, $m_k(G')\geq m_k(G)$ for each $k$.

Now consider a threshold graph on $n$ vertices with $e$ edges that has a bracketed 0-string, say $G = T(\sigma 10^j 1 \tau)$ for $\sigma$ and $\tau$ binary strings. Consider $G' = T(\sigma 010^{j-2} 10\tau)$. Figure \ref{fig:Bracketed0String}  demonstrates the difference between $G$ and $G'$ by showing the subgraph of $G$ induced by the bracketed 0-string on the left and the subgraph of $G'$ induced by the same vertices.

\begin{figure}[!ht]
\begin{center}
\begin{tikzpicture}
\coordinate (x) at (0,2);
\coordinate (y) at (0,0);
\coordinate (w) at (3,3);
\coordinate (v) at (3,2);
\coordinate (z) at (3,-1);

\draw (-.5,2) node {$x$};
\draw (-.5,0) node {$y$};
\draw (3.5,3) node {$w$};
\draw (3.5, 2) node {$v$};
\draw (3.5, -1) node {$z$};

\fill (x) circle (3pt);
\fill (y) circle (3pt);
\fill (w) circle (3pt);
\fill (z) circle (3pt);
\fill (v) circle (3pt);

\Edge (x)(z);
\Edge(x)(v);
\Edge(x)(w);

\Edge(x)(y);

\draw (1.5,-2) node {$1\underbrace{00\cdots 0}_j1$};

\coordinate (x2) at (5,2);
\coordinate (y2) at (5,0);
\coordinate (w2) at (8,3);
\coordinate (v2) at (8,2);
\coordinate (z2) at (8,-1);

\draw (4.5,2) node {$x$};
\draw (4.5, 0) node {$y$};
\draw (8.5, 3) node {$w$};
\draw (8.5, 2) node {$v$};
\draw (8.5, -1) node {$z$};

\fill (x2) circle (3pt);
\fill (y2) circle (3pt);
\fill (w2) circle (3pt);
\fill (z2) circle (3pt);
\fill (v2) circle (3pt); 

\Edge (y2)(z2);
\Edge(x2)(v2);
\Edge(x2)(z2);

\Edge(x2)(y2);

\draw (6.5, -2) node {$01\underbrace{00\cdots 0}_{j-2}10$};

\draw (3,1.5) ellipse (.75cm and 3cm);
\draw (8,1.5) ellipse (.75cm and 3cm);

\end{tikzpicture}
\caption{Removing a bracketed 0-string}
\label{fig:Bracketed0String}
\end{center}
\end{figure}
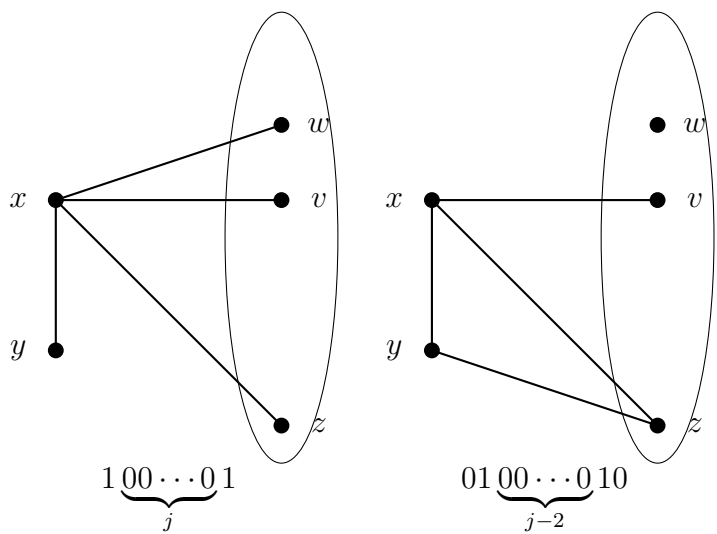

This time, the vertices in the oval represent the $0$'s in this segment of the threshold graph and therefore are an independent set.  Let $w$ be the vertex associated with the leftmost 0 in the bracketed 0-string and $z$ be the vertex associated with the rightmost $0$.  The $x$ represents the first $1$ and the $y$ is the second $1$.  Note $V(G)=V(G')$ and $E(G') = E(G) - xw + yz$.  

For a bracketed $0$-string, a similar injection from the matchings in $G$ to the matchings in $G'$ that is not a surjection can be defined.  
Define $r: E(G)\to E(G')$ by
$$r(e) =
\begin{cases}
yz &\text{if $e = xw$}\\
e\Delta\{y,x\} &\text{if $y\in e$ and $x\notin e$}\\
e\Delta\{z,w\} &\text{if $z\in e$ and $x\notin e$}\\
e  &\text{otherwise}
\end{cases}.$$

Define $\phi: \cM(G) \to \cM(G')$ by
$$\phi(M) = 
\begin{cases}
M &\text{if $xw\notin M$}\\
\{r(e): e\in M\} &\text{if $xw\in M$}\\
\end{cases}.$$

By a similar argument $\phi$ is an injection that preserves size and so $m(G)\leq m(G')$ and $m_k(G) \leq m_k(G')$ for all $k$.  Since  $M=\{xv, yz\}$ is not in the image of $\phi$ we know $G'$ has strictly more matchings than $G$.  Therefore, any graph that has a bracketed string does not have the maximum number of matchings.
\end{proof}

\subsubsection{Separation Issues}

Recall that a code has a separation issue if it has two pairs of repeated digits separated by a substring of an odd length with the first pair preceded by the opposite digit and the last pair not ending the code.  If there is a separation issue we can not write the code using $a$'s and $b$'s because we must separate each of the two sets of repeated digits to be an $a$ and a $b$, but we are left with a substring of odd length in the middle. So any graph with a separation issue is not almost alternating.

There are 4 possible types for the code of a threshold graph that has a separation issue depending on the value of the pairs of repeated digits. They are:
\begin{multicols}{2}
\begin{itemize}
\item $\cdots 011 \underbrace{\cdots}_{odd} 11\cdots *$
\item  $\cdots 011\underbrace{\cdots}_{odd} 00 \cdots *$ 
\end{itemize}
\begin{itemize}
\item $\cdots 100\underbrace{\cdots}_{odd} 00 \cdots *$ 
\item $\cdots 100\underbrace{\cdots}_{odd} 11 \cdots *$.
\end{itemize}
\end{multicols}

In the next lemma we will show that threshold graphs that have a separation issue do not maximize matchings.  The proof is by induction on the length of the substring in between the repeated pairs.  In the base case we will use an $ab$-switch to produce a bracketed string and then use Lemma \ref{bracketedstring} to conclude that the graph does not achieve the maximum.

\begin{lem} 
\label{separationissue} A threshold graph with a separation issue does not have the maximum number of matchings.
\end{lem}
\begin{proof}
We will only consider threshold graphs that have a separation issue and do not have a bracketed string since we have already proven that threshold graphs with bracketed strings do not have the maximum number of matchings. 
Define $\{a,b\}^*$ to be the set of all words formed with $a$'s and $b$'s.  By extension, we will talk about 0,1 strings being in $\{a,b\}^*$  if they can written as $a$'s and $b$'s.  We write, for example, $\{a,b\}^*a$ to mean a word in $\{a,b\}^*$ followed by an $a$.  Also, let $\epsilon$ be the empty word. 
   
  For the first case, suppose $G=T(\sigma 011 \underbrace{\cdots}_{odd} 11\tau)$ where $\sigma$ and $\tau$ are binary strings and $\tau$ is not the empty word.  Assume that the separation issue in the code of $G$ is minimal.  
  We claim that $G$ has a substring that looks like
   $$0110~w ~11$$
   for some $w\in\{a,b\}^*b\cup\{\epsilon\}$
   and the last digit shown here is not the last digit in the code.  
   
There are a couple of characteristics of this string that do not follow simply from the definition of a separation issue.  If there were not another $0$ following the first pair of repeated digits we would have a bracketed $1$-string.  By Lemma \ref{almostalternating} $w\in\{a,b\}^*$ since we are considering a minimal separation issue with no bracketed strings. 
The string must end with a $b$ to avoid a bracketed $1$-string with the second pair of repeated $1$'s.    
  
   Extending this idea to all cases, $G$ must have one of the following four strings in its code, 
\begin{enumerate}
\item  A $0110$ followed by $w\in\{a,b\}^*b\cup \{\epsilon\}$ followed by $11$.
$$0110~w~11$$

\item A $1001$ followed by $w\in\{a,b\}^*a\cup\{\epsilon\}$ followed by $00$.
$$1001~w~ 00$$

\item  A $0110$ followed by $w\in\{a,b\}^*a$ followed by $00$.
$$0110~w~00$$

\item  A $1001$ followed by $w\in\{a,b\}^*b$ followed by $11$.  
$$1001~w~11$$
\end{enumerate}

In all cases, there must be at least one more digit (or a $*$) to the right. The fact that $w$ can not be the empty word in the third and fourth cases is also necessary to avoid a bracketed string, as is the condition that $w$ end with a certain code.   

We claim that if these strings are part of the binary representation of the threshold graph with $n$ vertices and $e$ edges then the graph does not have the maximum number of matchings in $\cT_{n,e}$.

  The proof is by induction.  The proof of each case is similar, we will do cases 1 and 3 to illustrate the induction.
  
 \begin{enumerate}
 \item[1.]  
 
For the base case of the first case $w = \epsilon$ and so $G = T(\sigma 011011\tau)$ for strings $\sigma$ and $\tau$ with $\tau$ not empty.  Applying Lemma \ref{abvertedge}
 $$m(T(\sigma \underline{01}~\underline{10}~11 \tau) = m(T(\sigma 100111 \tau)).$$
Since $\tau$ is not empty we have a bracketed 1-string and can conclude the graph does not maximize matchings in $\cT_{n,e}$.  
Assume that a graph that has a substring of the form $0110~w~11$ is not maximal for all $w\in\{a,b\}^*b$ of length at most $n$.  Suppose $G = T(\sigma 0110 ~w~11\tau)$ where $\sigma$ and $\tau$ are binary strings with $\tau$ not empty and $w$ has length $n+1$.
 
   Suppose $w$ is a word of length $n+1$ that starts with $a$.  Let $w'$ be such that $w = 01 w'$.  So $w'$ has length $n$.  Then
 \begin{align*}
 m(T(\sigma~ 01~10~w~11~\tau))&=m(T(\sigma~ 01~\underline{10}~\underline{01}~w'~11~\tau))\\
 &= m(T(\sigma~ 01~ 01~10~w'~11~ \tau))
 \end{align*}
 using an $ab$-switch.  Now we have a separation issue of shorter length and so by the inductive hypothesis $G$ does not have the maximum number of matchings in $\cT_{n,e}$.
  
 Now consider the case where $w$ starts with a $b$.  Let $w'$ be such that $w = 10w'$.  So $w'$ has length $n$. Then 
 \begin{align*}
 m(T(\sigma~01~10~w~11~\tau))
 &=m(T(\sigma~\underline{01}~\underline{10}~10~w'~11~\tau))\\
 &= m(T(\sigma~10~ 01~10~w'~11~\tau)).
 \end{align*}
  Again we have reduced to a separation issue of shorter length. Thus any graph that contains the string in case 1 is not maximal.
  
 \item[3.]  
For the base case in the third case $w$ has length 1.  Since $w$  must end in $a$, we know that $w=a$.  Consider $G = T(\sigma 01100100\tau)$ where $\sigma$ and $\tau$ are binary strings and $\tau$ is not empty.  
Using an $ab$-switch,
$$m(T(\sigma ~01~\underline{10}~\underline{01}~00~\tau))= m(T(\sigma ~01~01~10~00~\tau)).$$
Recalling that $\tau$ is not empty we get a bracketed $0$-string and thus the graph does not have the  maximum number of matchings.  Now assume that any threshold graph of the form $ T(\sigma~0110~w~00~\tau)$ where $\sigma$ and $\tau$ are binary strings, $\tau$ is not empty, 
 and $w$ has length at most $n$ is not maximal.  Suppose $G=T(\sigma~0110~w~00~\tau)$ where $w\in\{a,b\}^*a$ has length $n+1$.  If $w$ starts with an $a$, let $w = aw'$ where $w'\in\{a,b\}^*a$ of length $n$. Then
\begin{align*}
m(T(\sigma ~01~10~w~00~\tau))&=m(T(\sigma~01~\underline{10}~\underline{01}~w'~00~\tau))\\
&=m(T(\sigma~ 01~01~10~w'~00~\tau))
\end{align*}
which has a shorter separation issue.  Similarly,  if $w$ begins with a $b$ write $w = bw'$ where $w'$ has length $n$. Then
\begin{align*}
m(T(\sigma~01~10~w~00~\tau))&=m(T(\sigma~\underline{01}~\underline{10}~10~w'~00~\tau))\\
&=m(T(\sigma~ 10~01~10~w'~00~\tau))
\end{align*}
which has a shorter separation issue.
 \end{enumerate}
\end{proof}

\subsection*{Proof of Theorem \ref{thm:MaxMatchings}}\label{proof:maxmatchings}

Recall that our main theorem concerning matchings states that a threshold graph has the maximum number of matchings among all graphs in $\cT_{n,e}$ if and only if $G$ is almost alternating.  


\begin{proof}
Suppose that a threshold graph $G$ is not  almost alternating. Then by Lemma \ref{almostalternating}, $G$ has a separation issue or a bracketed string.  By Lemma  \ref{bracketedstring} and Lemma \ref{separationissue}, $G$ does not have the maximum number of matchings.  Thus we conclude that almost alternating graphs maximize the number of matchings.  Since any graph that is not almost alternating has strictly fewer matchings and, by Corollary \ref{cor:SameNumber}, all almost alternating graphs have the same number of matchings we conclude that a threshold graph has the maximum number of matchings if and only if it is almost alternating.  
\end{proof}

\section{Minimizing Independent Sets}\label{sec:ind}

\subsection{Colex Graphs}\label{sec:colex}

Recall that colex graphs can be defined using the colex ordering on $E(K_n)$.  The following lemma gives an alternative characterization of threshold graphs using the binary code representation.

\begin{lem}\label{colex}
A graph is colex if and only if at most one $0$ appears after the first $1$ in the code.
\end{lem}
\begin{proof}
It is easy to see a graph is colex if and only if the vertex set can be partitioned into three sets: one on which the induced subgraph is a clique, one on which the induced subgraph is an independent set, and the third set consisting of at most one vertex that is incident only to vertices in the clique.  This corresponds to a binary code with an initial string of $0$'s followed by a string of $1$'s in which there may be one $0$ that corresponds to the vertex that is incident to some of the vertices in the clique.
\end{proof}

\subsection{Local Moves for Independent Sets}\label{sec:localmovesindsets}

In this section we prove two lemmas that, together, show that any threshold graph that has two $0$'s appearing after the first $1$ does not achieve the minimum number of independent sets.  Lemma \ref{bracketedstring2} deals with the case of consecutive $0$'s and Lemma \ref{twozeros} deals with the case of non-consecutive $0$'s. 
	
\begin{lem}\label{bracketedstring2}
For all $j\geq 2$, 

\[i(T(\sigma 1 0^j 1 \rho)) > i(T(\sigma 010^{j-2}10 \rho))\]
and for all $k\geq 0$, 

\[i_k(T(\sigma 1 0^j 1 \rho)) \geq i_k(T(\sigma 010^{j-2}10 \rho)).\]
\end{lem}
\begin{proof}

Consider the graphs $G = T(\sigma 10^j1\rho)$ and $G' = T(\sigma 010^{j-2}1 0 \rho)$ where $j\geq 2$ and $\sigma$ and $\rho$ are (possibly empty) binary strings.  By Lemma \ref{abvertedge}, $G$ and $G'$ have the same number of vertices and edges. Figure \ref{fig:Bracketed0String2} shows the subgraphs induced by the vertices not in $\sigma$ or $\rho$ of $G$ and $G'$ on the left and right, respectively.

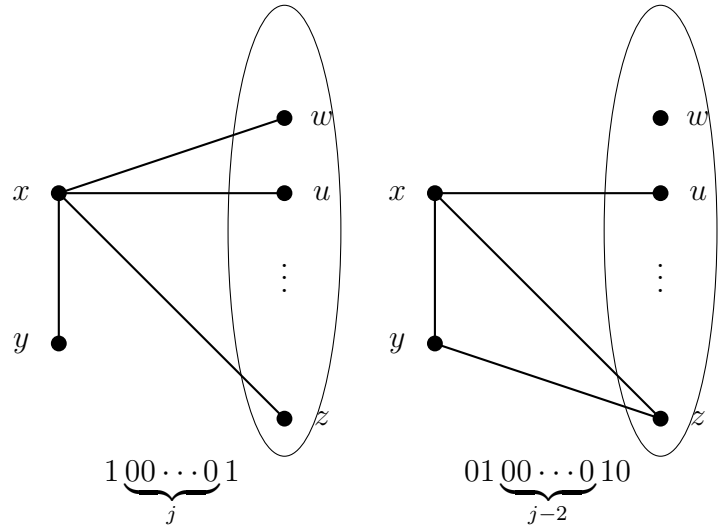
\begin{figure}[!ht]
\begin{center}
\begin{tikzpicture}
\coordinate (x) at (0,2);
\coordinate (y) at (0,0);
\coordinate (w) at (3,3);
\coordinate (v) at (3,2);
\coordinate (z) at (3,-1);

\draw (-.5,2) node {$x$};
\draw (-.5,0) node {$y$};
\draw (3.5,3) node {$w$};
\draw (3.5, 2) node {$u$};
\draw (3,1) node {$\vdots$};
\draw (3.5, -1) node {$z$};

\fill (x) circle (3pt);
\fill (y) circle (3pt);
\fill (w) circle (3pt);
\fill (z) circle (3pt);
\fill (v) circle (3pt);

\Edge (x)(z);
\Edge(x)(v);
\Edge(x)(w);

\Edge(x)(y);

\draw (1.5,-2) node {$1\underbrace{00\cdots 0}_j1$};

\coordinate (x2) at (5,2);
\coordinate (y2) at (5,0);
\coordinate (w2) at (8,3);
\coordinate (v2) at (8,2);
\coordinate (z2) at (8,-1);

\draw (4.5,2) node {$x$};
\draw (4.5, 0) node {$y$};
\draw (8.5, 3) node {$w$};
\draw (8.5, 2) node {$u$};
\draw(8,1) node {$\vdots$};
\draw (8.5, -1) node {$z$};

\fill (x2) circle (3pt);
\fill (y2) circle (3pt);
\fill (w2) circle (3pt);
\fill (z2) circle (3pt);
\fill (v2) circle (3pt); 

\Edge (y2)(z2);
\Edge(x2)(v2);
\Edge(x2)(z2);

\Edge(x2)(y2);

\draw (6.5, -2) node {$01\underbrace{00\cdots 0}_{j-2}10$};

\draw (3,1.5) ellipse (.75cm and 3cm);
\draw (8,1.5) ellipse (.75cm and 3cm);

\end{tikzpicture}
\caption{Removing consecutive $0$'s following the first $1$}
\label{fig:Bracketed0String2}
\end{center}
\end{figure}

We claim $i(G')< i(G)$.  Define $\phi: \mathcal{I}(G') \setminus \mathcal{I}(G) \to \mathcal{I}(G) \setminus \mathcal{I}(G')$ by 
\[\phi (A)  = A - \{x,w\} + \{y,z\}.\]  
Let $A\in\mathcal{I}(G') \setminus \mathcal{I}(G)$.  Note that $\{x,w\}\subset A$.  Consider $A' = A - \{x,w\} + \{y,z\}$.  Note that $y\notin A$ and $z\notin A$ since both $y$ and $z$ are incident to $x$ in $G'$.  Moreover, in $G$, $A'$ is an independent set.  To prove this we need only show that no vertex in $A'$ is incident to $y$ or $z$.  Say $v$ is incident to $y$. Then $v$ is to the right of $y$ or is to the left of $y$ and has code $1$.  Note that in both of these cases $v$ would be incident to $x$ in $G'$ and thus is not $A$.  Therefore no vertex incident to $y$ is in $A'$.  Similarly, if $v$ were incident to $z$ then $v$ must be a vertex to the left of $z$ with code $1$ and hence would also be incident to $w$.  Thus, $A' \in \mathcal{I}(G)$. Since $\{y,z\}\subset A'$ and $yz$ is an edge in $G'$ we know $A'\notin \mathcal{I}(G')$.  Since $\phi$ is clearly an injection that preserves size, $i(G') \leq i(G)$ and $i_k(G') \leq i_k(G)$.  Moreover, $\{y,w,z\}$ is not in the image of $\phi$ and so $i(G') <i(G)$.
\end{proof}

By Lemma \ref{bracketedstring2} any string of two or more $0$'s is part of the initial segment of $0$'s.  The only threshold graphs that remain that are not colex graphs have binary code $\rho101^j01\sigma$ for some binary strings $\sigma$ and $\rho$ and some $j\geq 1$. That is, two $0$'s must appear after the initial $1$ and these two $0$'s are not adjacent.

\begin{lem}\label{twozeros}
Let $\sigma$ and $\rho$ be binary strings. For all $j\geq 1$, $i(T(\rho101^j01\sigma)) >  i(T(\rho01^{j+2}0\sigma))$ and for all $k\geq 0$, $i_k(T(\rho101^j01\sigma)) \geq  i_k(T(\rho01^{j+2}0\sigma))$.
\end{lem}
\begin{proof}
Let $G = T(\rho101^j01\sigma)$ for $\rho$ and $\sigma$ (possibly empty) binary strings.  Let $G' = T(\rho01^{j+2}0\sigma)$.   By Lemma \ref{abvertedge}, $G'$ has the same number of vertices and edges as $G$. We claim $i(G')<i(G)$.   Figure \ref{fig:Bracketed1String} shows the subgraph of $G'$ induced by the vertices not in $\rho$ or $\sigma$ on the left.  On the right is the subgraph of $G$ induced by the same subset of the vertices.

\begin{figure}[!ht]
\begin{center}
\begin{tikzpicture}
\coordinate (x) at (0,2);
\coordinate (y) at (0,0);
\coordinate (w) at (3,3);
\coordinate (v) at (3,2);
\coordinate (z) at (3,-1);

\draw (3,1.5) ellipse (.75cm and 3cm);
\draw (8,1.5) ellipse (.75cm and 3cm);

\draw (3,3.75) node {$K_{j+2}$};
\draw (8,3.75) node {$K_{j+2}$};
\draw (0,3.75) node {$G'$};
\draw (5,3.75) node {$G$};

\draw (3,1) node {$\vdots$};
\draw (8,1) node {$\vdots$};

\draw (-.5,2) node {$x$};
\draw (-.5,0) node {$y$};
\draw (3.5,3) node {$w$};
\draw (3.5, 2) node {$v$};
\draw (3.5, -1) node {$z$};

\fill (x) circle (3pt);
\fill (y) circle (3pt);
\fill (w) circle (3pt);
\fill (z) circle (3pt);
\fill (v) circle (3pt);

\Edge (y)(z);
\Edge(y)(v);
\Edge(y)(w);

\draw (1.5,-2) node {$0\underbrace{11\cdots 1}_{j+2}0$};

\coordinate (x2) at (5,2);
\coordinate (y2) at (5,0);
\coordinate (w2) at (8,3);
\coordinate (v2) at (8,2);
\coordinate (z2) at (8,-1);

\draw (4.5,2) node {$x$};
\draw (4.5, 0) node {$y$};
\draw (8.5, 3) node {$w$};
\draw (8.5, 2) node {$v$};
\draw (8.5, -1) node {$z$};

\fill (x2) circle (3pt);
\fill (y2) circle (3pt);
\fill (w2) circle (3pt);
\fill (z2) circle (3pt);
\fill (v2) circle (3pt); 

\Edge (y2)(z2);
\Edge(y2)(v2);
\Edge(x2)(z2);

\draw (6.5, -2) node {$10\underbrace{11\cdots 1}_{j}01$};

\end{tikzpicture}
\caption{Removing non-consecutive $0$'s following the first $1$}
\label{fig:Bracketed1String}
\end{center}
\end{figure}
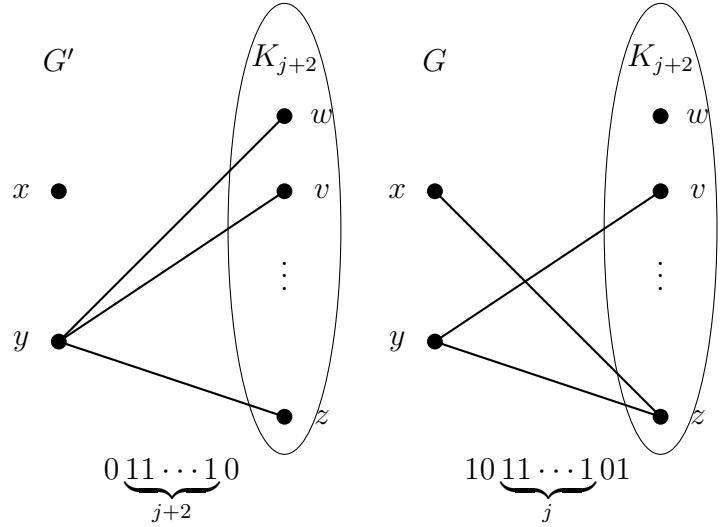

We will define an injection $\phi: \cI(G')\setminus \cI(G)\to \cI(G)\setminus \cI(G')$.  For $A\in \cI(G')\setminus \cI(G)$ define
$$\phi(A) = 
A - \{x,z\} + \{y,w\}.$$

Note that if $A\in \cI(G')\setminus \cI(G)$ then $\{x,z\} \subseteq A$ and $y\notin A$ and $w\notin A$. 
We need to show that $\phi(A) \in \cI(G) \setminus \cI(G')$.  
To show $\phi(A)\in \cI(G)$ we need only show that no vertices in $\phi(A)$ are incident to $y$ or $w$.  
If a vertex $v$ were incident to $y$ then $v$ corresponds to a $1$ to the left of $y$ and hence is incident to $z$.  
Therefore $v\notin A$ and so $v\notin \phi(A)$.  Suppose a vertex $v$ is incident to $w$.  
If $v$ corresponds to a $1$ to the left of $w$ then $v$ is also incident to $x$ and hence can't be in $A$.  
If $v$ is to the right of $w$ then $v$ must be incident to $z$.  
Therefore $\phi(A) \in \cI(G)$.  Since $\{y,w\}\in \phi(A)$ we know $\phi(A) \notin \cI(G')$.  

Clearly  $\phi$ is an injection that preserves size so $i(G') \leq i(G)$ and $i_k(G') \leq i_k(G)$.  Since $\{x,y,w\}$ is not in the image of $\phi$, $i(G')<i(G)$.
\end{proof}

We can use  Lemmas \ref{bracketedstring2}  and \ref{twozeros} to prove Theorem \ref{thm:MinIndSets} which states that $G\in\cT_{n,e}$ maximizes independent sets if and only if $G = \cC(n,e)$.

\begin{proof}[Proof of Theorem \ref{thm:MinIndSets}]
Let $G = T(\sigma)$ be a threshold graph with $n$ vertices and $e$ edges  that is not colex.  By Lemma \ref{colex}, $\sigma$ has more than one $0$ appearing after the initial $1$.  If $\sigma$ has a substring of the form $10^j1$ for some $j\geq 2$ then $G$ does not have the minimum number of independent sets by Lemma \ref{bracketedstring2}.  If $\sigma$ does not have a substring of this form, but does have more than one $0$ appearing after the initial 1 then $\sigma$ has a substring of the form $101^j01$ for some $j\geq 1$ (noting that the last $1$ may be the $*$).  In this case, $G$ does not have the minimum number of independent sets by Lemma \ref{twozeros}.  Therefore, the colex graph attains the minimum number of independent sets among threshold graphs and it is the only graph that does so.

Now we will show that $i_k(\cC(n,e)) \leq i_k(G)$ for all $k\geq 0$.  Among graphs $G$ that maximize $i_k(G)$, choose one with the binary number represented by the code smallest.  Suppose $G$ is not colex.  Since $G$ is not colex, the code of $G$ has more than $0$ appearing after the initial $1$ by Lemma \ref{colex}.  By Lemmas \ref{bracketedstring2} and \ref{twozeros} there exists $G'$ with at most as many $k$-independent sets.  Moreover, the binary number represented by the code of $G'$ in each case is less.  Therefore, $i_k(\cC(n,e))\leq i_k(G)$ for all $G$ with $n$ vertices and $e$ edges.
\end{proof}

\section*{Further Directions}

We have shown that the number of matchings (in total) is maximized by (any) almost alternating threshold graph and no other threshold graphs have the maximum number of matchings.  However, our results do not extend to matchings of size $k$.  This is, in particular, because our local move in Lemma \ref{bracketedstring} strictly increases the number of matchings in total, but does not necessarily strictly increase the number of matchings of size $k$.  Computer testing up to $22$ vertices supports the following conjecture.

\begin{conj}
If $k\geq 0$, $A$ is an almost alternating threshold graph, and $G$ is a threshold graph with the same number of vertices and edges then $m_k(G)\leq m_k(A)$. In addition, if $k\geq 2$, $m_k(A)>0$, and $G$ is not almost alternating then $m_k(G) < m_k(A)$.
\end{conj}

\section*{Acknowledgements}

The second author was sponsored by the Simons Foundation under Grant 429383.

\bibliographystyle{amsplain}
\bibliography{MaxMatchings}

\end{document}